\documentclass[a4paper,11pt,leqno]{amsart}

\usepackage{ulem}
\usepackage{a4wide}
\usepackage[english]{babel}
\usepackage[T1]{fontenc}
\usepackage[latin1]{inputenc}
\usepackage{dsfont}
\usepackage{mathrsfs}
\usepackage{amsmath}
\usepackage{amssymb}
\usepackage{verbatim}
\usepackage{amsthm}
\usepackage{graphicx}
\usepackage[usenames,dvipsnames,svgnames,table]{xcolor}
\usepackage[square,numbers]{natbib}
\usepackage{enumitem}
\usepackage{eurosym}

\def\a{\alpha}
\def\b{\beta}

\def\D{\Delta}
\def\e{\epsilon}

\def\t{\theta}
\def\T{\Theta}

\def\s{\sigma}

\def\o{\omega}

\def\cf{\textit{cf. }}
\def\NN{\mathbb N}

\def\RR{\mathbb R}

\def\fcar{\mathds{1}}

\def\esp{\mathbf E}

\def\prob{\mathbf P}

\def\nzeroun{\mathcal{N}(0,1)}

\theoremstyle{plain}
\newtheorem{theorem}{Theorem}
\newtheorem{lemma}{Lemma}
\newtheorem{proposition}{Proposition}
\newtheorem{corollary}{Corollary}
\newtheorem*{theorem*}{Theorem}
\newtheorem*{lemma*}{Lemma}
\newtheorem*{proposition*}{Proposition}
\newtheorem*{corollary*}{Corollary}
\theoremstyle{remark}

\newtheorem*{remark*}{Remark}

\newtheorem*{note*}{Note}
\theoremstyle{definition}

\newtheorem*{definition*}{Definition}

\begin{document}


\title{Optimal adaptive estimation of linear  functionals under sparsity}

\author{O. Collier, L. Comminges, A.B. Tsybakov and N. Verzelen}
\maketitle

\date{\today}

\begin{abstract} \
We consider the problem of estimation of a linear functional in the Gaussian sequence model where the unknown vector $\t\in \RR^d$ belongs to a class of $s$-sparse vectors with unknown $s$. We suggest an adaptive estimator achieving a non-asymptotic rate of 
convergence that differs from the minimax rate at most by a logarithmic factor. We also show that this optimal adaptive rate cannot be improved when $s$ is unknown. Furthermore, we address the issue of simultaneous adaptation to $s$ and to the variance $\s^2$ of the noise. We suggest an estimator that achieves the optimal adaptive rate when both $s$ and $\s^2$ are unknown. 
\end{abstract}



\section{Introduction}


We consider the model
\begin{equation}\label{model}
y_j= \t_j + \s\xi_j, \quad j=1,\dots,d,
\end{equation}
where $\t=(\t_1,\dots,\t_d)\in\RR^d$ is an unknown vector of parameters, $\xi_j$ are i.i.d. standard normal random variables, and $\s>0$ is the noise level. We study the problem of estimation of the linear functional
$$
L(\t)=\sum_{i=1}^d \t_i,
$$
based on the observations $y=(y_1,\dots,y_d)$. 

For $s\in \{1,\dots, d\}$, we denote by $\Theta_{s}$ the class  of all $\t\in\RR^d$ satisfying $\|\t\|_0 \le s$, where $\|\t\|_0$ denotes the number of non-zero components of $\t$. We assume that $\t$ belongs to $\Theta_{s}$ for some $s\in \{1,\dots, d\}$. Parameter $s$ characterizes the sparsity of vector $\theta$. 
The problem of estimation of $
L(\t)$ in this context arises, for example, if one wants to estimate the value of a function $f$ at a fixed point from noisy observations of its Fourier coefficients knowing that the function admits a sparse representation with respect to the first $d$ functions of the Fourier basis. Indeed, in this case the value $f(0)$ is equal to the sum of Fourier coefficients of $f$ with even indices.

As a measure of quality of an estimator $\hat T$ of the functional $L(\t)$ based on the sample $(y_1,\dots,y_d)$, we consider the maximum squared risk   
\begin{equation*}\label{definition_nonadaptive_risk}
\psi_{s}^{\hat{T}} \triangleq 
\sup_{\t\in \T_{s}} \esp_\t(\hat{T} - L(\t))^2,
\end{equation*}
where $\esp_\t$ denotes the expectation with respect to the distribution $\prob_\t$ of $(y_1,\dots,y_d)$ satisfying~\eqref{model}. For each fixed $s\in \{1,\dots, d\}$, the best quality of estimation is characterized by the minimax risk
\begin{equation*}
\psi^*_{s} \triangleq \inf_{\hat{T}} \sup_{\t\in \T_{s}} \esp_\t(\hat{T} - L(\t))^2,
\end{equation*}
where the infimum is taken over all estimators. An estimator $T^*$ is called rate optimal on $\T_s$ if $\psi_{s}^{T^*} \asymp \psi^*_{s}$. Here and in the following we write $a(d,s,\s) \asymp b(d,s,\s)$ for two functions $a(\cdot)$ and $b(\cdot)$ of $d, s$ and $\s$ if 
there exist absolute constants $c>0$ and $c'>0$ such that $c<a(d,s,\s)/b(d,s,\s)<c'$ for all $d$, all $s\in \{1,\dots,d\}$ and all $\s>0$.

The problem of estimation of the linear functional from the minimax point of view has been analyzed in  \cite{IbragimovHasminskii1984,CaiLow2004,CaiLow2005,Golubev2004, GolubevLevit2004,LaurentLudenaPrieur2008} among others. Most of these papers study minimax estimation of linear functionals on classes of vectors $\t$ different from $\T_s$. Namely, $\t$ is considered as a vector of first $d$ Fourier or wavelet coefficients of functions belonging to some smoothness class, such as Sobolev or Besov classes. In particular, the class of vectors $\t$ is assumed to be convex, which is not the case of class $\T_s$. \citet{CaiLow2004} were the first to address the problem of constructing rate optimal estimators of $L(\t)$ on  the sparsity class $\T_s$ and evaluating the minimax risk $\psi^*_{s}$. They studied the case $s<d^a$ for some $a<1/2$, with  $\sigma = 1/\sqrt{d}$, and established upper and lower bounds on $\psi^*_{s}$ that are accurate up to a logarithmic factor in $d$. The sharp non-asymptotic expression for the minimax risk $\psi^*_{s}$ is derived in  \cite{CollierCommingesTsybakov2015}  where it is shown that, for all $d$, all $s\in \{1,\dots,d\}$ and all $\s>0$
\begin{equation*}
\psi^*_{s} \asymp \s^2s^2\log(1+d/s^2).
\end{equation*}
Furthermore,  \cite{CollierCommingesTsybakov2015} proves that a simple estimator of the form
\begin{equation}\label{minimax_estimator}
\hat{L}_s^* = \begin{cases} \sum_{j=1}^d y_j \fcar_{y_j^2>2\s^2\log(1+d/s^2)}, &\text{ if } s<\sqrt{d}, \\
\sum_{j=1}^d y_j, &\text{ otherwise}, \end{cases}
\end{equation}
is rate optimal. Here and in the following, $\fcar_{\{\cdot\}}$ denotes the indicator function.

Note that the minimax risk $\psi^*_{s}$ critically depends on the parameter $s$ that in practice is usually unknown. More importantly, the rate optimal estimator $\hat{L}_s^* $ depends on $s$ as well, which makes it inaccessible in practice. 

In this paper, we suggest adaptive estimators of $L(\t)$ that do not depend on $s$ and achieve a non-asymptotic rate of 
convergence $\Phi^L(\s,s)$ that differs from the minimax rate $\psi^*_{s}$ at most by a logarithmic factor. We also show that this rate cannot be improved when $s$ is unknown in the sense of the definition that we give in Section \ref{sec:main} below. 
Furthermore, in Section~\ref{sec:unknown_sigma} we address the issue of simultaneous adaptation to $s$ and $\s$. We suggest an estimator that achieves the best rate  of adaptive estimation $\Phi^L(\s,s)$ when both $s$ and $\s$ are unknown.

\section{Main results}\label{sec:main}

Our aim is to show that the optimal adaptive rate of convergence is of the form 
$$
\Phi^L(\s,s) = \s^2s^2\log(1+d(\log d)/s^2)
$$
and to construct an adaptive estimator attaining this rate. 
Note that  
\begin{equation}\label{dlogd0}
\Phi^L(\s,s)\asymp \s^2 d(\log d), \quad \text{for all} \ \sqrt{d \log d}\le s \le d.
\end{equation}
Indeed, since the function $x\mapsto x\log(1+1/x)$ is increasing for $x>0$,
\begin{equation}\label{dlogd}
d(\log d)/2 \le s^2\log(1+d(\log d)/s^2)\le  d(\log d), \quad   \forall \ \sqrt{d \log d}\le s\le d, \ d\ge 3.
\end{equation}

To construct an adaptive estimator, we  first consider a collection of non-adaptive estimators  indexed by $s=1,\dots,d$:
\begin{equation}\label{definition_estimators_linear}
\hat{L}_s = \begin{cases} \sum_{j=1}^d y_j \fcar_{y_j^2>\a\s^2\log(1+d(\log d)/s^2)}, &\text{ if } s\leq\sqrt{d \log d/2} , \\
\sum_{j=1}^d y_j, &\text{ otherwise}, \end{cases}
\end{equation}
where $\a>0$ is a constant that will be chosen large enough. 
Note that if  in definition \eqref{definition_estimators_linear} we replace $d(\log d)$ by $d$, and $\a$ by $2$, we obtain the estimator $\hat{L}_s^*$ suggested in \cite{CollierCommingesTsybakov2015}, cf. \eqref{minimax_estimator}. It is proved in \cite{CollierCommingesTsybakov2015} that the estimator $\hat{L}_s^*$ is rate optimal in the minimax non-adaptive sense. 
The additional $\log d$ factor is necessary to achieve adaptivity as it will be clear from the subsequent arguments. 

We obtain an adaptive estimator via data-driven selection in the collection of estimators $\{\hat{L}_s \}$. The selection is based on a Lepski type scheme. For $s=1,\dots, d$, consider the thresholds $\o_s>0$ given by  
$$
\o_s^2 = \b \s^2 s^2 \log(1+d(\log d)/s^2) = \b \Phi^L(\s,s),
$$
where $\b>0$ is a constant that will be chosen large enough.
We define the selected index $\hat{s}$ by the relation
\begin{equation}\label{definition_hats}
\hat{s} \triangleq \min\Big\{s\in \{1,\dots,\lfloor \sqrt{d \log d/2}\rfloor\}: \,  |\hat{L}_s-\hat{L}_{s'}|\le \o_{s'} \ \text{for all} \ s'>s\Big\}
\end{equation}
with the convention that $\hat{s} =\lfloor \sqrt{d \log d/2}\rfloor+1$ if the set in \eqref{definition_hats} is empty. Here, $\lfloor \sqrt{d \log d/2}\rfloor$ denotes the largest integer less than $\sqrt{d \log d/2}$. 
Finally, we define an adaptive to $s$ estimator of $L$ as
\begin{equation}\label{definition_estimators_l0}
 \hat{L} \triangleq \hat{L}_{\hat{s}}.
 \end{equation}
The following theorem exhibits an upper bound on its risk.

\begin{theorem}\label{theorem_l0_upperbound}
Assume that $\a>48$, $\b\ge\frac{16}9(\sqrt{12}+2\sqrt{\a})^2$ and $d\ge d_0$, where $d_0\ge3$ is an absolute constant. Let $\hat{L}$~be the estimator defined in~(\ref{definition_estimators_l0}). Then, for all $\s>0$ and $s\in \{1,\dots,d\}$ we have
\begin{equation*}
\sup_{\t\in \T_{s}} \esp_\t(\hat{L} - L(\t))^2 \le C\Phi^L(\s,s)
\end{equation*}
for some absolute constant $C$.
\end{theorem}

Observe that for small $s$ (such that $s\leq d^{b}$ for $b<1/2$), we have $1\leq \Phi^L(\s,s)/\psi_{s}^*\leq c'$ where $c'>0$ is an absolute constant. 
Therefore, for such $s$ our estimator $\hat L$ attains the best possible rate on $\T_s$ given by the minimax risk $\psi_s^*$ and it cannot be improved, even by estimators depending on~$s$.  Because of this, the only issue is to check that the rate $\Phi^L(\s,s)$ cannot be improved if $s$ is greater than $d^b$ with $b<1/2$. For definiteness, we consider below the case $b=1/4$ but with minor modifications the argument applies to any $b<1/2$.
Specifically, we prove that any estimator whose maximal risk over $\T_s$ is  smaller (within a small constant) than $\Phi^L(\s,s)$ for some $s\ge d^{1/4}$, must have a maximal risk over $\T_1$ of power order in $d$ instead of the logarithmic order  $\Phi^L(\s,1)$
corresponding to our estimator. In other words, if we find an estimator that improves upon our estimator only slightly (by a  constant factor) for some $s\ge d^{1/4}$, then this estimator inevitably loses much more for small $s$, such as $s=1$, since there the ratio of maximal risks of the two estimators behaves as a power of~$d$.

\begin{theorem}\label{th:adaptation}
Let $d\ge 6$ and $\s>0$. There exist two small absolute constants $C_0>0$ and $C_1>0$ such that the following holds.  Any estimator $\widehat{T}$ that satisfies \[\sup_{\theta\in \T_s} \mathbf{E}_\theta\big[\big(\widehat{T}- L(\theta) \big)^2 \big]\leq C_0\Phi^L(\s,s) \quad \text{for some}\ s\geq d^{1/4} \]
has a degenerate maximal risk over $\T_1$, that is   
\[\sup_{\theta\in \T_1} \mathbf{E}_\theta\big[\big(\widehat{T}- L(\theta) \big)^2 \big]\geq C_1 \sigma^2 d^{1/4}\ .\] 
\end{theorem}

%

The property obtained in Theorem~\ref{th:adaptation} can be paraphrased in an asymptotic context to conclude that $\Phi^L(\s,s)$ is the adaptive rate of convergence on the scale of classes $\{\T_s, s=1,\dots, d\}$ in the sense of the definition in \cite{Tsybakov1998}. 
Indeed, assume that $d\to \infty$. Following  \cite{Tsybakov1998}, we call a function $s\mapsto \Psi_d(s)$ the {\it  adaptive rate of convergence on the scale of classes $\{\T_s, s=1,\dots, d\}$} if the following holds.
\begin{itemize}
\item[(i)] There exists an estimator $\hat{L}$ such that, for all $d$,
\begin{equation}\label{definition_optimality_criterion0}
 \max_{s=1,\dots,d}\ \sup_{\t\in \T_{s}} \esp_\t(\hat{L} - L(\t))^2/\Psi_d(s) \le C,
 \end{equation}
where $C>0$ is a constant (clearly, such an estimator $\hat{L}$ is adaptive since it cannot depend on $s$). 
\item[(ii)] If there exist another function $s\mapsto \Psi_d'(s)$ and a constant $C'>0$ such that, for all $d$,
\begin{equation}\label{definition_optimality_criterion00}
 \inf_{\hat{T}} \max_{s=1,\dots,d}\ \sup_{\t\in \T_{s}} \esp_\t(\hat{T} - L(\t))^2/\Psi_d'(s) \le C',
 \end{equation}
and 
\begin{equation}\label{def2}
\min_{s=1,\dots, d} \frac{\Psi_d'(s)}{\Psi_d(s)} \to 0 \ 
\text{as} \ d\to \infty,
\end{equation}
then there exists $\bar{s}\in \{1,\dots, d\}$ such that 
\begin{equation}\label{def3}
 \frac{\Psi_d'(\bar{s})}{\Psi_d(\bar{s})} \min_{s=1,\dots, d} \frac{\Psi_d'(s)}{\Psi_d(s)} \to \infty \ 
\text{as} \ d\to \infty.
\end{equation}
\end{itemize} 
In words, this definition states that the adaptive rate of convergence $\Psi_d(s)$ is such that any improvement of this rate for some $s$ (cf. \eqref{def2}) is possible only at the expense of much greater loss for another $\bar{s}$ (cf. \eqref{def3}).

\begin{corollary}\label{cor1}
The rate $\Phi^L(\s,s)$ is the adaptive rate of convergence on the scale of classes $\{\T_s, s=1,\dots, d\}$. 
\end{corollary}

It follows from the above results that the rate $\Phi^L(\s,s)$ cannot be improved when adaptive estimation on the family of sparsity classes $\{\T_s, s=1,\dots, d\}$ is considered. The ratio between the best rate  of adaptive estimation $\Phi^L(\s,s)$ and the minimax rate $\psi^*_s$ is equal to
$$
\phi^*_s= \frac{\Phi^L(\s,s)}{\psi^*_s}=\frac{\log(1+d(\log d)/s^2)}{\log(1+d/s^2)}.
$$
As mentioned above, $\phi^*_s\asymp 1$ if $s\leq d^{b}$ for $b<1/2$. In a vicinity of $s=\sqrt{d}$ we have $\phi^*_s\asymp \log\log d$, whereas for $s\ge \sqrt{d\log d}$ the behavior of this ratio is logarithmic: $\phi^*_s\asymp \log d$. Thus, there are different regimes and we see that, in some of them,   
rate adaptive estimation of the linear functional on the sparsity classes is impossible without loss of efficiency as compared to the minimax estimation.  However, this loss is at most logarithmic in $d$. 

We study now the adaptive rate of convergence on restricted scale of classes $\{\T_s,  d^{r_1} \leq s\leq d^{r_2}\}$ for some $0<r_1<r_2 \leq 1$. 

\begin{proposition}\label{prp:adaptation_restricted}
Fix $0<r_1< r_2\leq 1$. 
The adaptive rate of convergence on the scale of classes $\{\T_s,d^{r_1} \leq s\leq d^{r_2} \}$ is
 $\Phi^L(\s,s)$ if $r_1 <1/2$ and   $\sigma^2 d$ if $r_1 \geq 1/2$. 

\end{proposition}

\begin{proof}[Proof of Proposition \ref{prp:adaptation_restricted}]
For $r_1\geq 1/2$,  it is proved in \cite{CollierCommingesTsybakov2015} that the simple estimator $\widehat{L}^*_d=\sum_{j=1}^d y_j$ simultaneously achieves the minimax risk $\psi^*_s\asymp \sigma^2 d$ for all $s= \lfloor \sqrt{d} \rfloor,\ldots,d$. As a consequence, there is no loss for adaptation to the classes $\{\T_s, d^{1/2} \leq s\leq d \}$.

Now assume that $r_1<1/2$. In view of Theorem  \ref{theorem_l0_upperbound}, the estimator $\widehat{L}$ simultaneously achieves  the rate $\Phi^L(\s,s)$ for all classes  $\{\T_s,d^{r_1} \leq s\leq d^{r_2} \}$. It suffices to prove that this rate is optimal.  Below, $\lceil x\rceil$ stands for the smallest integer greater than or equal to $x$.
\begin{proposition}\label{lem:adaptation_general}
Fix $r_1\in (1/4, 1/2)$
Let $d\ge  6$ and $\s>0$. There exist two 
absolute constants $C_0>0$ and $C_1>0$ such that the following holds.  Any estimator $\widehat{T}$ that satisfies \[\sup_{\theta\in \T_s} \mathbf{E}_\theta\big[\big(\widehat{T}- L(\theta) \big)^2 \big]\leq C_0(1/2-r_1)\Phi^L(\s,s) \quad \text{for some}\ s\geq d^{(1/2+r_1)/2} \]
has a degenerate maximal risk over $\T_{\lceil d^{r_1}\rceil}$, that is   
\[\sup_{\theta\in \T_{\lceil d^{r_1}\rceil}} \mathbf{E}_\theta\big[\big(\widehat{T}- L(\theta) \big)^2 \big]\geq C_1(1/2-r_1) \sigma^2 d^{3r_1/2+ 1/2}\ .\] 
\end{proposition}
Note that $\Phi^L(\s, \lceil d^{r_1}\rceil )$ is not of larger order than $\s^2 d^{2r_1}\log(d)$, which is much smaller than $d^{3r_1/2+1/2}$. The proof of  Proposition \ref{lem:adaptation_general} follows immediately by applying Lemma \ref{proposition:adaptation} with $a=(r_1+1/2)/2$ and then concluding the proof as in  Corollary \ref{cor1}.

\end{proof}

\section{Adaptation to $s$ when $\s$ is unknown}\label{sec:unknown_sigma}

In this section we discuss a generalization of our adaptive estimator to the case when the standard deviation $\s$ of the noise  is unknown. 


To treat the case of unknown $\s$, we first construct an estimator $\hat \s$ of $\s$ such that, with high probability, $\s \le \hat \s \le 10\s$. Then, we consider the family of estimators  defined by a relation analogous to \eqref{definition_estimators_linear}:
\begin{equation}\label{definition_estimators_linear_unknown_sigma}
\hat{L}'_s = \begin{cases} \sum_{j=1}^d y_{j} \fcar_{y_{j}^2>\a{\hat \s}^2\log(1+d(\log d)/s^2)}, &\text{ if } s\leq\sqrt{d \log d/2}, \\
\sum_{j=1}^d y_{j}, &\text{ otherwise}, \end{cases}
\end{equation}
where $\a>0$ is a constant to be chosen large enough. The difference from \eqref{definition_estimators_linear} consists in the fact that we replace the unknown $\s$ by $\hat \s$. Then, we define a random threshold $\o_s'>0$ as
$$
(\o_s')^2 = \b \hat{\s}^2 s^2 \log(1+d(\log d)/s^2),
$$
where $\b>0$ is a constant to be chosen large enough. 
The selected index $\hat{s}'$ is defined by the formula analogous to \eqref{definition_hats}:
\begin{equation}\label{definition_hats_unknown_sigma}
\hat{s}' \triangleq \min\big\{s\in \{1,\dots,\lfloor \sqrt{d \log d/2}\rfloor\}: \,  |\hat{L}'_s-\hat{L}'_{s'}|\le \o_{s'}' \ \text{for all} \ s'>s\big\}.
\end{equation}
Finally, the adaptive estimator when $\s$ is unknown is defined as
$$
\hat{L}'\triangleq \hat{L}'_{\hat s'}.
$$
The aim of this section is to show that the risk of the estimator $
\hat{L}'$ admits an upper bound with the same rate as in Theorem~\ref{theorem_l0_upperbound} for all $d$ large enough. Consequently, $
\hat{L}'$ attains the best rate of adaptive estimation as follows from Section \ref{sec:main}.

Different estimators $\hat{\s}$ can be used.
By slightly modifying the method suggested in  \cite{CollierCommingesTsybakov2015}, we consider 
the statistic
\begin{equation}\label{definition_hat_sigma}
\hat{\s}=9 \Big(\frac{1}{\lfloor d/2\rfloor} \sum_{j\le d/2} y_{(j)}^2\Big)^{1/2}
\end{equation}
where $y_{(1)}^2\le \dots \le y_{(d)}^2$ are the order statistics associated to $y_1^2,\dots,y_d^2$. 
This statistic has the properties stated in the next proposition. In particular, $\hat{\s}$ overestimates $\s$ but it turns out to be without  prejudice to 
the attainment of the best rate by the resulting estimator $\hat{L}_s'$.

\begin{proposition} 
\label{prop1}
There exists an absolute constant $d_0\ge 3$ such that the following holds. Let $\hat{\s}$ be the estimator defined in \eqref{definition_hat_sigma}. Then, for all integers $d\ge d_0$ and $s<d/2$ we have
\begin{equation}\label{eq1_prop1}
\inf_{\t\in \T_s}\prob_{\t}(
\s \le \hat \s \le 10\s) \ge 1-d^{-5},
\end{equation}
and 
\begin{equation}\label{eq2_prop1}
\sup_{\t\in \T_s}\esp_{\t}(
{\hat \s}^4) \le {\bar C}\s^4,
\end{equation}
where ${\bar C}$ is an absolute constant.
\end{proposition}

The proof of this proposition is given in Section \ref{sec:upper}. Using Proposition \ref{prop1} we establish the following bound on the risk of the estimator $\hat{L}'$.

\begin{theorem}\label{th:unknown_sigma} Assume that $\a>48$, $\b\ge\frac{16}9(\sqrt{12}+2\sqrt{\a})^2$ and $d\ge d_0$ where $d_0\ge3$ is an absolute constant. Let $\hat{\s}$ be the estimator defined in \eqref{definition_hat_sigma}.   
Then, for the estimator $\hat{L}'$ with tuning parameters $\a$ and $\b$, for all $\s>0$, and all $s<d/2$ we have
\begin{equation}\label{eq:upp_unknown_sigma}
\sup_{\t\in \T_{s}} \esp_\t(\hat{L}' - L(\t))^2 \le C\Phi^L(\s,s)
\end{equation}
for some absolute constant $C$.
\end{theorem}
Thus, the estimator $\hat{L}'$, which is independent of both $s$ and $\s$ achieves the rate $\Phi^L(\s,s)$ that is the best possible rate of adaptive estimation established in Section \ref{sec:main}.

The condition $s< d/2$ in this theorem can be generalized to $s\leq \zeta d$ for some $\zeta \in (0,1)$. In fact, for any $\zeta \in (0,1)$, we can modify the definition of  \eqref{definition_hat_sigma} by summing only over the $(1-\zeta)d$ smallest values of $y_i^2$. Then, changing the numerical constants $\alpha$ and $\beta$ in the definition of $\omega'_s$, we obtain that the corresponding estimator $\hat{L}'$ achieves the best possible rate simultaneously for all $s\leq \zeta d$ with a constant $C$ in \eqref{eq:upp_unknown_sigma} that would depend on $\zeta$. However, we cannot set $\zeta=1$.
Indeed,  the following proposition shows that it is not possible to construct an estimator, which is simultaneously adaptive to all $\sigma>0$ and to all $s\in [1,d]$. 

\begin{proposition}\label{th:adaptation_sigma}
Let $d\ge 3$ and $\s>0$. There exists a small absolute constant $C_0>0$ such that the following holds. 
Any estimator $\widehat{T}$ that satisfies 
\begin{equation}\label{eq:T1_sigma_adaptative}
\sup_{\theta\in \T_1} \mathbf{E}_\theta\big[\big(\widehat{T}- L(\theta) \big)^2 \big]\leq C_0 \sigma^2 d \ ,\quad \quad \forall \s>0,
\end{equation}
has a degenerate maximal risk over $\T_d$, that is,  for any fixed  $\sigma>0$, 
\begin{equation}\label{eq:T2_sigma_adaptative}
\sup_{\theta\in \T_d}   \mathbf{E}_\theta\big[\big(\widehat{T}- L(\theta) \big)^2 \big]=\infty \ . 
\end{equation}

\end{proposition}

In other words, when $\sigma$ is unknown, any estimator, for which the maximal risk over $\T_d$ is finite for all $\s$, cannot achieve over $\T_1$ a risk of smaller order than $\sigma^2d$, and hence cannot be minimax adaptive. Indeed, as shown above, the adaptive minimax rate over $\Theta_1$ is of the order $\sigma^2\log d$.

\section{Proofs of the upper bounds}\label{sec:upper}

In the following, we will denote $c_1,c_2,\ldots$ absolute positive constants and write for brevity $L$ instead of $L(\t)$. 

\subsection{Proof of Theorem~\ref{theorem_l0_upperbound}}

Let $s\in \{1,\dots,d\}$ and assume that $\t$ belongs to $\Theta_s$.  We have
\begin{equation}\label{decomposition_upperbound_linear_1}
\esp_\t (\hat{L}-L)^2 = \esp_\t \big[ (\hat{L}_{\hat{s}}-L)^2 \fcar_{\hat{s}\le s} \big] +  \esp_\t \big[ (\hat{L}_{\hat{s}}-L)^2 \fcar_{\hat{s}>s} \big].
\end{equation}
Consider the first summand on the right hand side of~(\ref{decomposition_upperbound_linear_1}). Set for brevity $s_0=\lfloor\sqrt{d \log d/2}\rfloor+1$.
Using the definition of $\hat{s}$ we obtain, on the event $\{\hat{s}\le s\}$,
\begin{equation*}
(\hat{L}_{\hat{s}}-L)^2 \le 2 \o^2_{s} + 2(\hat{L}_{s}-L)^2 \ \text{if} \ s<s_0 \ \text{or} \ s\ge s_0, \hat{s}<s_0.
\end{equation*}
Thus,
\begin{eqnarray}\label{proof_upper1}
\forall \ s< s_0: \quad \esp_\t \big[ (\hat{L}_{\hat{s}}-L)^2 \fcar_{\hat{s}\le s} \big] &\le& 2\beta \Phi^L(\s,s) + 2\esp_\t (\hat{L}_{s}-L)^2 ,
\\  \label{proof_upper1_bis}
\forall \ s \ge  s_0: \quad   \esp_\t \big[ (\hat{L}_{\hat{s}}-L)^2 \fcar_{\hat{s}\le s} \big] &\le& 
\esp_\t \big[ (\hat{L}_{\hat{s}}-L)^2 (\fcar_{\hat{s}\le s, \hat{s}<s_0} + \fcar_{\hat{s}=s_0} )\big] 
\\
&\le&
2\beta \Phi^L(\s,s) + 2\esp_\t (\hat{L}_{s}-L)^2 + \esp_\t (\hat{L}_{s_0}-L)^2. \nonumber
\end{eqnarray}
By Lemma \ref{lem:risk} proved at the end of this section, we have
\begin{equation*}
\sup_{\t\in \Theta_s} \esp_\t (\hat{L}_{s}-L)^2 \le c_1\Phi^L(\s,s), \quad s=1,\dots,s_0 -1.
\end{equation*}
Note that, in view of \eqref{dlogd0}, for all $s\in[s_0,d]$  we have 
$$
\Phi^L(\s,s_0) \le \s^2d \log d\le 2\s^2s^2\log(1+(d\log d)/s^2)= 2 \Phi^L(\s,s),
$$ 
and by definition of $\hat{L}_{s}$, for all $s\in[s_0,d]$ and all $\t\in \RR^d$, we have $\esp_\t (\hat{L}_{s}-L)^2\le\s^2d\le 2 \Phi^L(\s,s)$. 
Combining these remarks with \eqref{proof_upper1} and \eqref{proof_upper1_bis} yields
\begin{equation}\label{proof_upper2}
\sup_{\t\in \Theta_s} \esp_\t \big[ (\hat{L}_{\hat{s}}-L)^2 \fcar_{\hat{s}\le s} \big] \le c_2 \Phi^L(\s,s), \quad s=1,\dots,d.
\end{equation}
Consider now the second summand on the right hand side of~(\ref{decomposition_upperbound_linear_1}). Since $\hat{s}  \leq s_0$ we obtain the following two facts. First,
\begin{equation}\label{proof_upper2a}
\sup_{\t\in \Theta_s}\esp_\t \big[ (\hat{L}_{\hat{s}}-L)^2 \fcar_{\hat{s}>s} \big] =0, \quad \forall \ s\ge  s_0.
\end{equation}
Second, on the event $\{\hat{s}>s\}$,
\begin{equation*}
(\hat{L}_{\hat{s}}-L)^4 \le \sum_{s<s'\leq s_0} (\hat{L}_{s'}-L)^4.
\end{equation*}
Thus,
\begin{eqnarray}\nonumber
\sup_{\t\in \Theta_s}\esp_\t \big[ (\hat{L}_{\hat{s}}-L)^2 \fcar_{\hat{s}>s} \big] &\le& \sup_{\t\in \Theta_s} \Big[\sqrt{\prob_\t(\hat{s}>s)} (d \log d)^{1/4}\max_{s<s'\leq s_0}\sqrt{\esp_\t(\hat{L}_{s'}-L)^4}\Big]\\
&\le& (d \log d)^{1/4} \sup_{\t\in \Theta_s} \sqrt{\prob_\t(\hat{s}>s)} \max_{s'\leq s_0} \Big[\sup_{\t\in \Theta_{s'}} \sqrt{\esp_\t(\hat{L}_{s'}-L)^4}\Big]
\label{inequality_upperbound_1}
\end{eqnarray}
where for the second inequality we have used that $\Theta_{s}\subset \Theta_{s'}$ for $s<s'$.
To evaluate the right hand side of \eqref{inequality_upperbound_1} we use the following two lemmas.
\begin{lemma}\label{lemma_power4}
Recall the definitions of $\hat{L}_s$ and $\hat{L}'_s$ in~(\ref{definition_estimators_linear}) and~(\ref{definition_estimators_linear_unknown_sigma}). For all $s\leq s_0=\lfloor\sqrt{d \log d/2}\rfloor+1$, we have
\begin{align*}
&\sup_{\t\in\T_{s}} \esp_\t\big(\hat{L}_{s}-L\big)^4 \le c_3 \s^4 d^4 (\log d)^2, \\
&\sup_{\t\in\T_{s}} \esp_\t\big(\hat{L}_{s}'-L\big)^4 \le c_4 \s^4 d^4 (\log d)^2.
\end{align*}
\end{lemma}
\begin{lemma}\label{lemma_probability}
Assume that $\a>48$ and $\b= \frac{16}9(\sqrt{12}+2\sqrt{\a})^2$. \\
(i) We have
\begin{equation}\label{eq:lem2_1}
\max_{s\leq \sqrt{d \log d/2}} \ \sup_{\t\in\T_s} \prob_\t(\hat{s}>s) \le c_5 d^{-5}.
\end{equation}
(ii) We have
\begin{equation*}
\max_{s\leq \sqrt{d \log d/2}} \ \sup_{\t\in\T_s} \prob_\t(\hat{s}'>s) \le c_6 d^{-5}.
\end{equation*}
\end{lemma}
From~\eqref{proof_upper2a}, (\ref{inequality_upperbound_1}), the first inequality in Lemma \ref{lemma_power4}, and part (i) of Lemma~\ref{lemma_probability}  
 we find that
\begin{equation*}\label{inequality_upperbound_2}
\sup_{\t\in \Theta_s}\esp_\t \big[ (\hat{L}_{\hat{s}}-L)^2 \fcar_{\hat{s}>s} \big] \le \sqrt{c_3c_5}  \s^2 \le c_7 \Phi^L(\s,s), \quad s=1,\dots,d.
\end{equation*}
Combining this inequality with \eqref{decomposition_upperbound_linear_1} and \eqref{proof_upper2} we obtain the theorem.

\subsection{Proofs of the lemmas} 

\begin{proof}[Proof of Lemma~\ref{lemma_power4}]
For $s=s_0$, $\hat{L}_s - L= \hat{L}'_s - L=\s\sum_{i=1}^d\xi_i$. As a consequence 
\[
 \esp_\t (\hat{L}_s - L)^4 = \esp_\t (\hat{L}'_s - L)^4 = 3\s^4 d^2\leq 3\s^4 d^4(\log d)^2\ .
\]
Henceforth, we focus on the case $s\leq \sqrt{d\log(d)/2}$. We have
\begin{equation}\label{eq:lem1}
\hat{L}_s - L = \s\sum_{i=1}^d \xi_i - \sum_{i=1}^d y_i \fcar_{y_i^2\le \a\s^2\log(1+d(\log d)/s^2)}.
\end{equation}
Thus, 
\begin{equation*}
\esp_\t (\hat{L}_s - L)^4 \le 8\Big( \s^4\esp\Big(\sum_{i=1}^d \xi_i \Big)^4 +d^4 \a^2\s^4\log^2(1+d(\log d)/s^2)\Big)\le c_3 \s^4 d^4 (\log d)^2.
\end{equation*}
In a similar way, 
\begin{equation}\label{eq:lem1_2}
\hat{L}_s' - L = \s\sum_{i=1}^d \xi_{i} - \sum_{i=1}^d y_{i} \fcar_{y_{i}^2\le \a{\hat\s}^2\log(1+d(\log d)/s^2)},
\end{equation}
and
\begin{equation*}
\esp_\t (\hat{L}_s' - L)^4 \le 8\Big( \s^4\esp\Big(\sum_{i=1}^d \xi_{i} \Big)^4 +d^4 \a^2\esp_\t({\hat\s}^4)\log^2(1+d(\log d)/s^2)\Big).
\end{equation*}
The desired bound for $\esp_\t (\hat{L}_s' - L)^4$  follows from this inequality and \eqref{eq2_prop1}.
\end{proof}

\begin{proof}[Proof of Lemma~\ref{lemma_probability}] 
We start by proving part (i) of Lemma~\ref{lemma_probability}. Note first that, for $s \leq \sqrt{d \log d/2}$ and all $\theta$ we have
\begin{equation}\label{eq:lem2_1a}
\prob_\t\big(|\hat{L}_{s}-L|>3\o_{s'}/4\big)\le \prob_\t\big(|\hat{L}_{s}-L|>3\o_{s}/4\big), \quad \forall \ s< s'\le  d.
\end{equation}
Indeed, if $s<s'$ we have $\o_{s'}>\o_{s}$ since the function $t\mapsto \o_t$ is increasing for $t>0$. Thus
$$
\prob_\t\big(|\hat{L}_{s'}-\hat{L}_{s}|>\o_{s'}\big)\le \prob_\t\big(|\hat{L}_{s'}-L|>3\o_{s}/4\big) + \prob_\t\big(|\hat{L}_{s}-L|>\o_{s}/4\big).
$$
This inequality and the definition of $\hat{s}$ imply that, for all $s \leq \sqrt{d \log d/2}$ and all $\theta$, 
\begin{align}\label{eq:lem2_2a}
\prob_\t(\hat{s}>s) \le & \sum_{s<s'\le d} \prob_\t\big(|\hat{L}_{s'}-\hat{L}_{s}|>\o_{s'}\big) \\
\le & ~d \prob_\t\big(|\hat{L}_{s}-L|>3\o_{s}/4\big)+ \sum_{s<s'\le d}  \prob_\t \big(|\hat{L}_{s'}-L|>\o_{s'}/4\big).\nonumber
\end{align}
Note that, for $\sqrt{d \log d/2}<  s'\le d$,  we have $\hat{L}_{s'}=\sum_{i=1}^d y_i$, and $\o_{s'}\ge \s\sqrt{\b d \log d}\sqrt{\log(3)/2}$ by monotonicity. Hence, for $\sqrt{d \log d/2}<  s'\le d$, and all $\theta$,
$$
\prob_\t \big(|\hat{L}_{s'}-L|>\o_{s'}/4\big)\le \prob \Big(\big|\sum_{i=1}^d \xi_i\big|>\frac{\sqrt{\b d \log d}}{4}\sqrt{\log(3)/2}\Big)\le 2d^{-\b\log(3)/64},
$$
where we have used that $\xi_i$ are i.i.d. standard Gaussian random variables. This inequality and \eqref{eq:lem2_2a} imply that, for $s  \le \sqrt{d \log d/2}$, and all $\theta$,
 \begin{align}\label{eq:lem2_2b}
\prob_\t(\hat{s}>s) 
\le & ~ \sqrt{d \log d/2} \max_{s<s'\leq  \sqrt{d \log d}}  \prob_\t \big(|\hat{L}_{s'}-L|>3\o_{s'}/4\big)\\
&~+d \prob_\t\big(|\hat{L}_{s}-L|>3\o_{s}/4\big) + 2d^{-\b\log(3)/64}.\nonumber
\end{align}
As $\T_s\subset \T_{s'}$ for $s<s'$, we have
$$
 \max_{s<s'\leq  \sqrt{d \log d}} \ \sup_{\t\in\T_s} \prob_\t \big(|\hat{L}_{s'}-L|>3\o_{s'}/4\big)\le \max_{s'\leq  \sqrt{d \log d}} \ \sup_{\t\in\T_{s'}} \prob_\t \big(|\hat{L}_{s'}-L|>3\o_{s'}/4\big).
$$
Together with \eqref{eq:lem2_2b} this implies
 \begin{align*}
\max_{s \leq \sqrt{d \log d}} \ \sup_{\t\in\T_s}  \prob_\t(\hat{s}>s) 
\le & ~ 2d \max_{s'\leq  \sqrt{d \log d}} \ \sup_{\t\in\T_{s'}} \prob_\t \big(|\hat{L}_{s'}-L|>3\o_{s'}/4\big)
+ 2d^{-\b\log(3)/64}. 
\end{align*}
Considering the assumption on $\b$, the last summand in this inequality does not exceed $2d^{-5}$.
Thus,  it remains to bound the first term in the right-hand side.

Fix $s\leq \sqrt{d \log d/2}$ and let $\t$ belong to $\T_s$. We will denote by  $S$ the  support of $\t$ and we set for brevity 
$$
a\triangleq \sqrt{\log(1+d(\log d)/s^2)}.
$$
From \eqref{eq:lem1} and the fact that $y_i=\t_i+\s\xi_i$ we have
\begin{align}\label{eq:lem2_3a}
|\hat{L}_s - L| = &\Big|\s\sum_{i\in S} \xi_i - \sum_{i\in S} y_i \fcar_{y_i^2\le \a\s^2a^2} + \s\sum_{i\not \in S} \xi_i \fcar_{\xi_i^2> \a a^2}\Big|\\
\le & \s \Big|\sum_{i\in S} \xi_i\Big| + \s\Big|\sum_{i\not \in S} \xi_i \fcar_{\xi_i^2> \a a^2}\Big|+  \sqrt{\a}\s s a . \nonumber
\end{align}
Recalling that $\o_s=\sqrt{\b}\s s a$ we find
\begin{align}\label{eq:lem2_3} 
\prob_\t \big(|\hat{L}_{s}-L|>3\o_{s}/4\big) &\le \prob \Big(\Big|\sum_{i\not \in S} \xi_i \fcar_{\xi_i^2> \a a^2}\Big|> \sqrt{\a}s a\Big) \\
&+
\prob \Big(\Big|\sum_{i \in S} \xi_i \Big|> (3\sqrt{\b}/4-2\sqrt{\a})s a\Big) \nonumber
\end{align}
 Since $\xi_i$ are i.i.d. $\nzeroun$ random variables, we have
\begin{equation}\label{eq:lem2_4}
\prob \Big(\Big|\sum_{i \in S} \xi_i \Big|> (3\sqrt{\b}/4-2\sqrt{\a})s a\Big) \le 2 \exp\Big(-\frac{(3\sqrt{\b}/4-2\sqrt{\a})^2}{2} s a^2 \Big).
\end{equation}
We now use the relation
\begin{equation}\label{eq:lem2_5}
s a^2= s \log(1+d(\log d)/s^2) \ge \log d \quad \text{for all} \ s\in [1,\sqrt{d \log d/2}],
\end{equation}
since the function $s\to s\log(1+d\log(d)/s^2)$ is increasing. It follows from \eqref{eq:lem2_4}, \eqref{eq:lem2_5} and the assumption on $\a$ and $\b$ that
\begin{equation}\label{eq:lem2_6}
\prob \Big(\Big|\sum_{i \in S} \xi_i \Big|> (3\sqrt{\b}/4-2\sqrt{\a})s  a\Big) \le 2d^{-6}.
\end{equation}

Next, consider the first probability on the right hand side of \eqref{eq:lem2_3}. To bound it from above,
we invoke the following lemma. 

\begin{lemma}\label{lemma_truncated_gaussian}
If $\a>48$, for all  $s\leq \sqrt{d \log d/2}$ and all  $U\subseteq \{1,\dots, d\}$,
\begin{equation*}
\prob \Big(\sup_{t\in[1,10]}\Big|\sum_{i\in U} \xi_i \fcar_{|\xi_i|> \sqrt{\a} a t}\Big|> \sqrt{\a} s a\Big)\le c_8d^{-6}.
\end{equation*}
\end{lemma}
Combining \eqref{eq:lem2_3},  \eqref{eq:lem2_6} and Lemma~\ref{lemma_truncated_gaussian}  we obtain part (i) of Lemma \ref{lemma_probability}.

\smallskip

We now proceed to the proof of part (ii) of Lemma \ref{lemma_probability}. Proposition \ref{prop1} implies that, for $s \leq \sqrt{d \log d}$ and $\t\in\T_s$,
$$
\prob_\t(\hat{s}'>s)\le \prob_\t(\hat{s}'>s,  \hat{\s}\in[\s,10\s]) +  \prob_\t(\hat{\s}\not\in[\s,10\s])   .
$$
On the event $\{ \hat{\s}\in[\s,10\s]\}$, we can replace $\hat{\s}$ in the definition of $\hat{s}'$ either by $\s$ or by $10\s$ according to cases, thus making the analysis of $\prob_\t(\hat{s}'>s, \hat{\s}\in[\s,10\s])$ equivalent, up to the values of numerical constants, to the analysis of $\prob_\t(\hat{s}>s)$ given below. The only non-trivial difference consists in the fact that the analog of \eqref{eq:lem2_3a}
when $\hat{L}_{s}$ is replaced by $\hat{L}_{s}'$ contains the term  $\s\Big|\sum_{i\not \in S} \xi_i \fcar_{\xi_i^2> \a \hat{\s}^2a^2/\s^2}\Big|$ instead of $\s\Big|\sum_{i\not \in S} \xi_i \fcar_{\xi_i^2> \a a^2}\Big|$ while $\hat{\s}$ depends on $\xi_1,\dots,\xi_d$. This term is evaluated using Lemma~\ref{lemma_truncated_gaussian} and the fact that 
\begin{equation*}
\prob \Big(\Big|\sum_{i\not \in S} \xi_i \fcar_{|\xi_i|>  \sqrt{\a} \hat{\s} a/\s}\Big|> \sqrt{\a} s a, \ \hat{\s}\in[\s,10\s] \Big)
\le \prob \Big(\sup_{t\in[1,10]}\Big|\sum_{i\not \in S} \xi_i \fcar_{|\xi_i|> \sqrt{\a} a t}\Big|> \sqrt{\a} s a\Big).
\end{equation*}
We omit further details that are straightforward from inspection of the proof of part (i) of Lemma~\ref{lemma_probability} given above. Thus, part (ii) of Lemma \ref{lemma_probability} follows.
\end{proof}

%

For the proof of Lemma~\ref{lemma_truncated_gaussian}, recall the following fact about the tails of the standard Gaussian distribution, which can be proven by integration by part. 
\begin{lemma}\label{lemma_gaussian}
Let $X\sim\nzeroun$,  $x>1$ and  $q\in\NN$. There is a constant $C_q^*$ such that
\begin{equation*}
\esp\Big[ X^{2q}\fcar_{|X|>x} \Big] \le C_q^* x^{2q-1} e^{-x^2/2}.
\end{equation*}
Moreover, simulations suggest that $C_1^*\le1.1$.
\end{lemma}
We will also use the Fuk-Nagaev inequality \cite[page 78]{petrov} that we state here for reader's convenience.
\begin{lemma}[Fuk-Nagaev inequality]\label{lemma:fuk}
Let $p>2$ and $v>0$. Assume that $X_1,\ldots,X_n$ are independent random variables with $\esp (X_i)=0$  and $\esp |X_i|^p<\infty$, $i=1,\dots, n$. Then,
\begin{equation*}
\prob\Big(\sum_{i=1}^n X_i > v \Big) \le (1+2/p)^p \sum_{i=1}^n \esp |X_i|^p v^{-p} + \exp\left(- \frac{2 v^2}{(p+2)^2e^{p} \sum_{i=1}^n \esp X_i^2} \right).
\end{equation*}
\end{lemma}

\begin{proof}[Proof of Lemma~\ref{lemma_truncated_gaussian}] We have 
\begin{eqnarray*}\label{eq1:lemma_truncated_gaussian_bis}
p_0 &\triangleq& \prob \Big(\sup_{t\in[1,10]}\big|\sum_{i\in U} \xi_i \fcar_{|\xi_i|> \sqrt{\a} a t}\big|> \sqrt{\a} s a\Big) 
\\
&=&
 \esp \Big[\prob \Big(\sup_{t\in[1,10]}\big|\sum_{i\in U} \e_i |\xi_i |\fcar_{|\xi_i|> \sqrt{\a} a t}\big|> \sqrt{\a} s a \,\Big| \ |\xi_i |, i\in U \Big)\Big]
\end{eqnarray*}
where $\e_i$ denotes the sign of $\xi_i$. Consider the function  $g(x) = \sup_{t\in[1,10]}\big|\sum_{i\in U} x_i |\xi_i |\fcar_{|\xi_i|> \sqrt{\a} a t}\big|$ where $x=(x_i, i\in U)$ with $x_i\in \{-1,1\}$. For any $i_0\in U$, let $g_{i_0,u}(x)$ denote the value of this function when we replace $x_{i_0}$ by $u\in \{-1,1\}$. Note that, for any fixed $(|\xi_i |, i\in U)$, we have the bounded differences condition:
$$
\sup_{x}|   g(x)- g_{i_0,u}(x)|
\le 2|\xi_i| \fcar_{|\xi_i|> \sqrt{\a} a } \triangleq 2Z_i
\quad \forall u \in \{-1,1\}, i_0\in U.
$$
The vector of Rademacher random variables $(\e_1,\dots,\e_d)$ is independent from $(|\xi_1|,\dots,|\xi_d|)$.Thus, for any fixed $(|\xi_i |, i\in U)$ we can use the bounded differences inequality
, which yields
\begin{eqnarray}\label{eq2:lemma_truncated_gaussian_bis}
\quad\quad
p_0 \le 
 \esp \Big[\exp\Big(- \frac{\a s^2a^2}
 {2 \sum_{i\in U} Z_i^2 }  \Big)\Big]
 \le  \exp\Big(- \frac{\a s^2a^2}
 {2 \D }  \Big) + \prob\Big( \sum_{i\in U} Z_i^2 >\D\Big), \quad \forall \ \D>0.
\end{eqnarray}
We now set $\D=  \sum_{i\in U} \esp Z_i^2  + d\a a^2\exp\left(-\a a^2/(2p)\right)$ with $p=\a/8>6$.

To bound from above the probability $\prob\Big( \sum_{i\in U} Z_i^2 >\D\Big)$ we apply Lemma~\ref{lemma:fuk} with $X_i=Z_i^2  - \esp( Z_i^2 )$ and $v= \a a^2 d\exp\left(-\a a^2/(2p)\right)$.
The random variables $X_i$ are centered and satisfy, in view of Lemma~\ref{lemma_gaussian},
\begin{equation}\label{fuk1}
\esp |X_i|^p \le 2^{p-1}\esp  |Z_i|^{2p} \le  2^{p-1}C_{p}^*(\sqrt{\a}a)^{2p-1} e^{-\a a^2/2}
\end{equation}
Thus, Lemma~\ref{lemma:fuk} yields
\begin{eqnarray*}
\prob\Big( \sum_{i\in U} Z_i^2 >\D\Big)
 &\le&   C_{p}^* 2^{p-1}(1+2/p)^p \frac{(\sqrt{\a}a)^{-1}}{d^{p}}  + \exp\left(- \frac{ \sqrt{\a}a d \exp(\a a^2(1/2-1/p))}{2(p+2)^2e^{p}C_2^*} \right).
\end{eqnarray*}
The expression in the last display can be made smaller than $c_9d^{-6}$ for all $d\ge 3$.

Finally, using \eqref{fuk1} we find 
\begin{align*}\label{expon1}
\frac{\a s^2a^2}
 {2 \D } &\ge \frac{\a s^2a^2}{2d (C_1^*\sqrt{\a} a \exp(-\a a^2/2)+ \a a^2 \exp(-\a a^2/(2p)))}\ge  \frac{s^2 \exp(\a a^2/(2p))}{4.4d},
\end{align*}
whereas
$$
\frac{s^2 \exp(\a a^2/(2p))}{d}=
\frac{s^2}{d}  \Big(1+ \frac{d\log d}{s^2}\Big)^{\a/(2p)}  = \log d \Big(\frac{s^2}{d\log(d)}+1\Big) \times \Big(1+\frac{d\log d}{s^2}\Big)^{\a/(2p)-1}\ge  3^3\log d
$$
for any $s\leq\sqrt{d \log d/2}$, since $\a= 8p$. Hence, for such $s$,
\begin{align*}
\exp\Big(- \frac{\a s^2a^2}
 {2 \D }  \Big)  \le c_{10}d^{-6}.
\end{align*} 
Thus, Lemma~\ref{lemma_truncated_gaussian} follows.
\end{proof}

\begin{lemma}\label{lem:risk} There exists an absolute constant $d_0\ge3$ such that if $\a>48$, we have
\begin{equation*}
\sup_{\t\in \Theta_s} \esp_\t (\hat{L}_{s}-L)^2 \le c_1\Phi^L(\s,s), \quad \sup_{\t\in \Theta_s} \esp_\t (\hat{L}_{s}'-L)^2 \le c_{11}~\Phi^L(\s,s), \quad \forall s\leq\sqrt{d\log d/2}.
\end{equation*}
\end{lemma}
\begin{proof}  We easily deduce from \eqref{eq:lem2_3a} that
\begin{align*}
\esp_\t (\hat{L}_{s}-L)^2 \le  3 \s^2\Big( s + d \esp \left[ X^2 \fcar_{X^2> \a a^2}\right]+  \a s^2 a^2\Big),
\nonumber
\end{align*}
where $X\sim \nzeroun$. By Lemma~\ref{lemma_gaussian}, \begin{equation*}\label{eq:fukk}
d \esp \left[ X^2 \fcar_{X^2> 2 a^2}\right]\le C_1^* ad \exp(-a^2) = \frac{C_1^* ad s^2}{s^2+d\log d} \le \frac{C_1^*s^2 a}{\log d},
\end{equation*}
which implies that the desired bound for $\esp_\t (\hat{L}_{s}-L)^2$ holds since $\a\ge2$. Next, we prove the bound of the lemma for $\esp_\t (\hat{L}_{s}'-L)^2$. Similarly to \eqref{eq:lem2_3a},
\begin{align}\label{eq:lem2_3aa}
\hat{L}_s' - L = &\s\sum_{i\in S} \xi_{i} - \sum_{i\in S} y_{i} \fcar_{y_{i}^2\le \a\hat{\s}^2a^2} + \s\sum_{i\not \in S} \xi_{i} \fcar_{\s^2\xi_{i}^2> \a \hat{\s}^2 a^2}.
\nonumber
\end{align}
This implies
\begin{eqnarray}\label{uuu}
\quad \esp_\t \left[(\hat{L}_s' - L)^2\fcar_{\hat{\s}\in[\s,10\s]}\right]
&\le & \ \esp_\t\Big(\s\Big|\sum_{i\in S} \xi_{i}\Big| +  \sqrt{\a}\hat{\s} s a + \s W \Big)^2\\
&\le&   \
3 \Big( \s^2 s +  \a \esp_\t(\hat{\s}^2) a^2 s^2+  \s^2  \esp (W^2)\Big),\nonumber
\end{eqnarray}
where $W\triangleq \sup_{t\in[1,10]}\Big|\sum_{i\not \in S} \xi_{i} \fcar_{|\xi_{i}|> \sqrt{\a} a t}\Big|$.
Using Lemma~\ref{lemma_truncated_gaussian} we find that, for all $\a>48$, 
\begin{eqnarray*}\label{}
	\esp (W^2)
	&\le & \ (\sqrt{\a}sa)^2 + \esp\Big(\sum_{i\not \in S} |\xi_{i}| \Big)^2 \fcar_{W> \sqrt{\a} s a}
	\\
	&\le & \ \a s^2 a^2 + \Big[\esp\Big(\sum_{i\not \in S} |\xi_{i}| \Big)^4 \Big]^{1/2} c_9 d^{-3}
	\le  \a s^2 a^2 + c_9\sqrt{3}d^{-1}.
\end{eqnarray*}
Plugging this bound in \eqref{uuu} and using \eqref{eq2_prop1} we get 
$$
\esp_\t \left[(\hat{L}_s' - L)^2\fcar_{\hat{\s}\in[\s,10\s]}\right]
\le  c_{12} \Phi^L(\s,s).
$$
On the other hand, by virtue of Lemma \ref{lemma_power4} and \eqref{eq1_prop1},
\begin{align*}
\esp_\t \left[(\hat{L}_s' - L)^2\fcar_{\hat{\s}\not\in[\s,10\s]}\right]\le \sqrt{\prob_\t(\hat{\s}\not\in[\s,10\s])}\sqrt{\esp_\t (\hat{L}_s' - L)^4}\le \frac{\sqrt{c_3}\s^2\log d}{d^{1/2}}\le c_{13} \Phi^L(\s,s).
\end{align*}
The desired bound for $\esp_\t (\hat{L}_{s}'-L)^2$ 
follows from the last two displays. 

\end{proof}

\subsection{Proofs of Proposition \ref{prop1} and of Theorem \ref{th:unknown_sigma}} 
\begin{proof}[Proof of Proposition \ref{prop1}] 
Since $s\leq d/2$, there exists a subset  $T$ of size $\lfloor d/2\rfloor $ such that $T\cap S=\emptyset$.
By Definition of $\hat{\s}^2$, we obtain that 
$$\hat{\s}^2\leq \frac{81\sigma^2}{\lfloor d/2\rfloor }\sum_{i\in T} \xi_i^2\ .$$
This immediately implies  \eqref{eq2_prop1}. To prove \eqref{eq1_prop1}, note that the Gaussian concentration inequality (\cf\cite{LedouxTalagrand1991}) yields
$$\prob\Big(\Big(\sum_{i\in T} \xi_i^2\Big)^{1/2} > \sqrt{100\lfloor d/2\rfloor/81  }\Big)\le \exp(-cd)\ , $$
for a positive constant $c$. 
Therefore,
\begin{equation}\label{eq:proof:prop1}
\prob_\t(\hat{\s} \le 10  \s)\ge 1-\exp(-c d).
\end{equation}
Next, let $\mathcal{G}$ be the collection of all subsets of $\{1,\dots,d\}$ of cardinality $\lfloor d/2\rfloor$. We now establish a bound on the deviations of random variables $Z_G = \frac{1}{\s^2}\sum_{i\in G} y_i^2$ uniformly over all $G\in \mathcal{G}$. Fix any $G\in \mathcal{G}$. The random variable $Z_G$ has a chi-square distribution with $\lfloor d/2\rfloor$ degrees of freedom and non-centrality parameter $\sum_{i\in G}\theta_i^2$. In particular, this distribution is stochastically larger than a  central chi-square distribution with $d' = \lfloor d/2\rfloor$ degrees of freedom. Let $Z$ be a random variable with this central chi-square  distribution. For the tail probability of $Z$, we can use Lemma 11.1 in \cite{verzelenTGD} that gives
\[
 \mathbf{P}\Big(Z \leq \frac{d'}{e}x^{2/d'}\Big)\leq x, \qquad \forall \ x>0.
\]
Take $x= \binom{d}{d'}^{-1}e^{- d'/2}$. Using the bound $\log \binom{d}{d'}\leq d'\log(ed/d')$ it follows that $\log(1/x)\leq d'(\tfrac{3}{2}+ \log(\tfrac{d}{d'}))\leq d'(\tfrac{3}{2}+\log 2) + 1$. Taking the union bound over all $G\in \mathcal{G}$ we conclude that 
\[
 \mathbf{P}\bigg(\inf_{G\in \mathcal{G}}Z_G \leq \frac{d'}{4e^3}\Big(1-\frac{2}{d'}\Big)\bigg)\leq e^{-d'/2}< d^{-5}/2\ 
\]
for all $d$ large enough. Since $\widehat{\s}^2= \sigma^2    \frac{81}{d'} \inf_{G\in \mathcal{G}} Z_G^2$, we obtain that $\widehat{\sigma}^2\geq \sigma^2$ with probability at least $1-d^{-5}/2$ for all $d$ large enough. Combining this with  \eqref{eq:proof:prop1}, we get  \eqref{eq1_prop1} for all $d$ large enough.

\end{proof}

\begin{proof}[Proof of Theorem \ref{th:unknown_sigma}] 
We repeat the proof of Theorem~\ref{theorem_l0_upperbound} replacing there $\hat{L}_s$ by $\hat{L}_s'$  and  $\hat{s}$ by $\hat{s}'$. The difference is that, in view of \eqref{eq2_prop1}, the relation \eqref{proof_upper1}
now holds with   $c_{14}\beta \Phi^L(\s,s)$ instead of $\beta \Phi^L(\s,s)$, and we use the results of Lemmas \ref{lemma_power4}, \ref{lemma_probability} and \ref{lem:risk} related to $\hat{L}_s'$ rather than to~$\hat{L}_s$.  
\end{proof}

\section{Proofs of the lower bounds}

\subsection{Proof of  Theorem \ref{th:adaptation}}

Theorem \ref{th:adaptation} is an immediate consequence of the following lemma with $a=1/4$.
\begin{lemma}\label{proposition:adaptation}
For all $d\geq 6$, $a\in [1/4,1/2)$,   and  $s\geq d^{a}$, 
\begin{equation}\label{eq:definition_R_s}
R(s)\triangleq \inf_{\tilde{L}}\left\{\esp_0 (\tilde{L}-L)^2 \s^{-2}d^{-3a+1/2}+ \sup_{\t\in\T_s}\esp_\t (\tilde{L}-L)^2\big(\Phi^L(\s,s)\big)^{-1}\right\}\ge \frac{1/2 -a}{40}.
\end{equation}
\end{lemma}
\begin{proof}
We first introduce some notation. For a probability measure $\mu$ on $\Theta_s$, we denote by $\mathbb{P}_{\mu}$ the mixture probability measure $\mathbb{P}_{\mu}= \int_{\Theta_s}\mathbf{P}_{\theta}\mu( d\theta)$.  Let $\mathcal{S}(s,d)$ denote the set of all subsets of $\{1,\ldots, d\}$ of size $s$, and let $S$ be a set-valued random
variable uniformly distributed on $\mathcal{S}(s, d)$. For any $\rho>0$, denote by $\mu_{\rho}$ the distribution of the
random variable $\sigma \rho\sum_{j\in S} e_j$ where $e_j$ is the $j$th canonical basis vector in $\mathbb{R}^d$. Next, let
$\chi^2(Q,P)= \int (dQ /dP )^2 dP-1$
denote the chi-square divergence between two  probability measures $Q$ and~$P$ such that $Q\ll P$, and $\chi^2(Q,P)=+\infty$ if $Q\not\ll P$.

Take any $a\in [1/4,1/2)$ and  $s\geq d^{a}$. Set  
 \[
 \rho \triangleq \sqrt{(1/2-a)\log(1+d(\log d)/s^2)}=\sqrt{1/2-a}\big(\Phi^L(\s,s)\big)^{1/2}/(s\sigma).
 \] 
Consider the mixture distribution $\mathbb{P}_{\mu_{\rho}}$ with this value of $\rho$.
For any estimator $\tilde{L}$, we have  $\sup_{\t\in\T_s}\esp_\t (\tilde{L}-L)^2\ge \mathbb{E}_{\mu_{\rho}}(\tilde{L} - L)^2\ge \mathbb{E}_{\mu_{\rho}}(\tilde{L} - \mathbb{E}_{\mu_{\rho}}(L))^2 =\mathbb{E}_{\mu_{\rho}}(\tilde{L} - \s s\rho)^2$. Therefore, 
 \begin{eqnarray} \nonumber
  R(s)&\geq&\inf_{\tilde{L}}\left\{\mathbf{E}_{0}(\tilde{L} ^2) \s^{-2}d^{-3a+1/2}+ \mathbb{E}_{\mu_{\rho}}(\tilde{L} - \s s\rho)^2 \big(\Phi^L(\s,s)\big)^{-1}\right\} \\ \nonumber
 &\geq &  \frac{1/2-a}{4} \inf_{\tilde{L}} \left\{\mathbf{P}_0(\tilde{L}> \s s\rho /2)\s^{-2}d^{-3a+1/2}\Phi^L(\s,s)+ \mathbb{P}_{\mu_{\rho}}(\tilde{L}< \s s\rho /2)\right\}\\
 &\geq &  \frac{1/2-a}{4} \inf_{\mathcal{A}} \left\{\mathbf{P}_0(\mathcal{A})\s^{-2}d^{-3a+1/2}\Phi^L(\s,s)+ \mathbb{P}_{\mu_{\rho}}(\mathcal{A}^c)\right\},
 \label{eq:lower_tilde_L}
 \end{eqnarray}
where $\inf_{\mathcal{A}}$ denotes the infimum over all measurable events $\mathcal{A}$, and $\mathcal{A}^c$ denotes the complement of $\mathcal{A}$. It remains to prove that the expression in \eqref{eq:lower_tilde_L} is not smaller than $(1/2-a)/40$. This will be deduced from the following lemma, the proof of which is given at the end of this section.
\begin{lemma}\label{lem:unbalanced}
Let $P$ and $Q$ be two probability measures on a measurable space $({X},{\mathcal U})$. Then, for any $q>0$,
\begin{equation*}
\inf_{\mathcal{A}\in {\mathcal U}} \left\{P(\mathcal{A})q+ Q(\mathcal{A}^c)\right\}\ge \max_{0<\tau<1}\left[\frac{q\tau}{1+ q\tau}\big(1-\tau (\chi^2(Q,P)+1)\big)\right].
\end{equation*}
\end{lemma}

We now apply Lemma~\ref{lem:unbalanced} with $P=\mathbf{P}_0$, $Q=\mathbb{P}_{\mu_{\rho}}$, and 
\begin{equation}\label{q}
q=\s^{-2}d^{-3a+1/2}\Phi^L(\s,s)=s^2d^{-3a+1/2} \log\Big(1+ \frac{d(\log d)}{s^2}\Big).
\end{equation}
By Lemma 1 in \cite{CollierCommingesTsybakov2015}, the chi-square divergence $\chi^2(\mathbb{P}_{\mu_{\rho}},\mathbf{P}_{0})$ satisfies 
\[\chi^2(\mathbb{P}_{\mu_{\rho}},\mathbf{P}_{0})\leq \left(1- \frac{s}{d}+ \frac{s}{d}e^{\rho^2}\right)^{s} -1\leq \left(1+\frac{s}{d}\left(e^{\rho^2}-1\right)\right)^{s}\ .\]
Since $\rho^2=(1/2-a) \log\big(1+ \frac{d(\log d)}{s^2}\big)$, we find
\begin{eqnarray}
\label{chi}
\chi^2(\mathbb{P}_{\mu_{\rho}},\mathbf{P}_{0}) &\leq & \exp\left[s\log\left[1+ \frac{s}{d}\left(\Big(1+ \frac{d(\log d)}{s^2}\Big)^{1/2-a} -1\right)\right]\right] \\ \nonumber
&\leq & \exp\left[s\log\left(1+ (1/2-a)\frac{\log d}{s}\right)\right]\leq d^{1/2-a}\ , 
\end{eqnarray}
where we have used that   $(1+x)^{1/2-a}\le 1+(1/2-a)x$ for $x>0$. Take
 \begin{equation}\label{tt}
 \tau=(d^{1/2-a}+1)^{-1}/2. 
 \end{equation}
 Then, using \eqref{q} and the inequality $s\ge d^{a}$ we find
\begin{equation}\label{qt}
q\tau = \frac{s^2\log\Big(1+ \frac{d(\log d)}{s^2}\Big)}{2d^{3a-1/2} (d^{1/2-a}+1)}\ge \frac{d^{2a}\log(1+ d^{1-2a}(\log d))}{2d^{3a-1/2} (d^{1/2 -a}+1)}>\frac14, \quad \forall \ d\ge 6.
\end{equation}
Lemma~\ref{lem:unbalanced} and inequalities \eqref{chi} -- \eqref{qt} imply
\begin{eqnarray} \nonumber
  \inf_{\mathcal{A}} \left\{\mathbf{P}_0(\mathcal{A})\s^{-2}d^{-3a+1/2}\Phi^L(\s,s)+ \mathbb{P}_{\mu_{\rho}}(\mathcal{A}^c)\right\} \ge \frac{q\tau}{2(1+ q\tau)}\ge \frac1{10}.
 \label{eq:lower_tilde_L_1}
 \end{eqnarray}
\end{proof}

\begin{proof}[Proof of Lemma \ref{lem:unbalanced}] We follow the same lines as in the proof of Proposition 2.4 in \cite{Tsybakov2009}. Thus,  for any $\tau\in (0,1)$,
$$
P(\mathcal{A}) \ge \tau (Q(\mathcal{A}) - v), \quad \text{where} \ v= Q\left(\frac{dP}{dQ}<\tau \right) \le \tau (\chi^2(Q,P)+1).
$$
Then,
\begin{eqnarray*}
\inf_{\mathcal{A}} \left\{P(\mathcal{A})q+ Q(\mathcal{A}^c)\right\}&\ge &
\inf_{\mathcal{A}} \left\{q\tau (Q(\mathcal{A}) - v)+ Q(\mathcal{A}^c)\right\}
\\
&\ge& 
\min_{0\le t\le 1} \max (q\tau (t-v), 1-t) = \frac{q\tau(1-v)}{1+ q\tau}.
\end{eqnarray*}

\end{proof}

\subsection{Proof of Corollary \ref{cor1}}
First, note that condition \eqref{definition_optimality_criterion0} with $\Psi_d(s) = C\Phi^L(\s,s)$ is satisfied due to Theorem~\ref{theorem_l0_upperbound}. Next, the minimum in condition \eqref{def2} with $\Psi_d(s)= C\Phi^L(\s,s)$ can be only attained for $ s\ge d^{1/4}$, since for $ s< d^{1/4}$ we have $\Phi^L(\s,s)\asymp \psi^*_s$  where $ \psi^*_s$ is the minimax rate on $\T_s$. Thus, it is not possible to achieve a faster rate than $\Phi^L(\s,s)$ for $ s< d^{1/4}$, and therefore  \eqref{def2} is equivalent to the condition
$$
 \min_{s\ge d^{1/4}} \frac{\Psi_d'(s)}{\Phi^L(\s,s)} \to 0,
$$
and
$$
 \min_{s=1,\dots, d} \frac{\Psi_d'(s)}{\Phi^L(\s,s)} \asymp  \min_{s\ge d^{1/4}} \frac{\Psi_d'(s)}{\Phi^L(\s,s)}.
$$
Obviously, $\Psi_d'(s)$ cannot be of smaller order than the minimax rate $\psi^*_s$, which implies that 
$$
 \min_{s\ge d^{1/4}} \frac{\Psi_d'(s)}{\Phi^L(\s,s)} \ge \min_{s\ge d^{1/4}} \frac{c\psi^*_s}{\Phi^L(\s,s)}=\min_{s\ge d^{1/4}}\frac{c\log(1+d/s^2)}{\log(1+d(\log d)/s^2)}\ge \frac{c'}{\log d}
$$
where $c,c'>0$ are absolute constants. On the other hand, Theorem~\ref{th:adaptation} yields 
$$
\frac{C'\Psi_d'(1)}{\Phi^{L}(\s,1)}\geq  \frac{C'C_1\sigma^2d^{1/4}}{\Phi^{L}(\s,1)}= \frac{C'C_1d^{1/4}}{\log(1+ d(\log d))}\ .
$$
Combining the last three displays, we find 
$$
\frac{\Psi_d'(1)}{\Phi^{L}(\s,1)} \min_{s=1,\dots, d} \frac{\Psi_d'(s)}{\Phi^L(\s,s)} \ge \frac{c'C'C_1d^{1/4}}{(\log d)\log(1+ d(\log d))}\to \infty, 
$$
as $d\to \infty$, thus proving \eqref{def3} with $\bar{s}=1$.

\subsection{Proof of Proposition \ref{th:adaptation_sigma}}

Since in this proof we consider different values of $\s$, we denote the probability distribution of $(y_1,\ldots, y_d)$ satisfying~\eqref{model} by ${\bf P}_{\theta,\sigma^2}$. Let ${\bf E}_{\theta,\sigma^2}$ be the corresponding expectation. Assume that $\widehat{T}$ satisfies \eqref{eq:T1_sigma_adaptative} with $C_0=1/512$.  We will prove that \eqref{eq:T2_sigma_adaptative} holds for $\sigma=1$. The extension to arbitrary $\sigma>0$ is straightforward and is therefore omitted.

Let $a>1$ be a positive number and let $\mu$ be the $d$-dimensional normal distribution with zero mean and covariance matrix $a^2 {\bf I}_d$ where ${\bf I}_d$ is the identity matrix. In what follows, we consider the mixture probability measure $\mathbb{P}_{\mu}=\int_{\Theta_d}\mathbf{P}_{\theta,1}\mu( d\theta)$. Observe that $\mathbb{P}_{\mu}={\bf P}_{0,1+a^2}$.

Fixing $\theta=0$ and $\sigma^2=1+a^2$ in \eqref{eq:T1_sigma_adaptative}, we get $\mathbf{E}_{0,1+a^2}\big[\widehat{T}^2\big]\leq 2C_0a^2 d$ and therefore
$\mathbf{P}_{0,1+a^2}(|\widehat{T}|\geq  \tfrac{1}{8} a\sqrt{d}  )\leq \frac{1}{4}$. Since $\mathbb{P}_{\mu}={\bf P}_{0,1+a^2}$, this implies
\begin{equation}\label{eq:pmu1}
\mathbb{P}_{\mu}\big(|\widehat{T}|<  \frac{1}{8} a\sqrt{d}  \big)>\frac{3}{4}.  
\end{equation}
For $\theta$ distributed according to $\mu$, $L(\theta)$ has a normal distribution with mean 0 and variance $a^2d$. Hence, using the table of standard normal distribution, we find 
\[
 \mu\Big(|L(\theta)|\leq  \frac{a}{4} \sqrt{d}\Big)< \frac{1}{4}\ .
\]
Combining this with \eqref{eq:pmu1}, we conclude that, with $\mathbb{P}_{\mu}$-probability greater than $1/2$, we have simultaneously $|L(\theta)|> a\sqrt{d}/4$ and $|\widehat{T}|<  a \sqrt{d}/8$. Hence, 
\[
 \sup_{\theta\in \T_d}  \mathbf{E}_{\theta,1}\big[\big(\widehat{T}- L(\theta) \big)^2 \big]\geq  \mathbb{E}_{\mu}\big[(\widehat{T}-L(\theta))^2\big]\geq  \frac{1}{128} a^2 d
\]
where $\mathbb{E}_{\mu}$ denotes the expectation with respect to $\mathbb{P}_{\mu}$.  The result now follows by letting $a$ tend to infinity.

\vspace{5mm}

{\bf Acknowledgement.}
The work of A.B.Tsybakov was supported by GENES and by the French National Research Agency (ANR) under the grants
IPANEMA (ANR-13-BSH1-0004-02) and Labex Ecodec (ANR-11-LABEX-0047). It was also supported by the "Chaire Economie et Gestion des Nouvelles Donn\'ees", under the auspices of Institut Louis Bachelier, Havas-Media and Paris-Dauphine. The work of O. Collier has been conducted as part of the project Labex MME-DII (ANR11-LBX-0023-01).

\end{document}